\documentclass[twoside, 12pt]{article}
\usepackage[russian, english]{babel}
\usepackage{epsfig}
\usepackage{amssymb,amsmath,amsfonts,amsthm,enumerate}
\usepackage{amsfonts}
\usepackage{amsmath}
\usepackage{graphicx}
\usepackage{amsthm}
\usepackage[cp1251]{inputenc}
\usepackage[T2A]{fontenc}
\usepackage{mathrsfs}
\usepackage[english]{babel}
 \newcommand{\grant}[1]{\medskip \baselineskip 10pt{\footnotesize #1} \medskip}

\newtheorem{theorem}{Theorem}[section]

\newtheorem{remark}{Remark}
\newtheorem{proposition}{Proposition}[section]

\numberwithin{equation}{section}
 \def\@evenhead{\vbox{\hbox to \textwidth{\thepage\hfil\sl\leftmark\strut}\hrule}}
 \def\@oddhead{\vbox{\hbox to \textwidth{\rightmark\hfill\thepage\strut}\hrule}}

\begin{document}
 \sloppy

\centerline{\bf ONE  PROPERTY OF ZERO SET OF FUNCTION INVERTIBLE}
\centerline{\bf IN THE SENSE OF EHRENPREIS IN THE SCHWARTZ ALGEBRA}     

\vskip 0.3cm

\centerline{\bf  N.~Abuzyarova, A. Idrisova, K. Khasanova}        

\vskip 0.3cm

\vskip 0.7 cm


\vskip 0.2cm


\vskip 0.2cm

\noindent {\bf Abstract.}
 
We consider those elements of the Schwartz algebra of entire functions which are Fourier-Laplace transforms
 of invertible distributions with compact supports on the real line. These functions are called invertible 
in the sense of Ehrenpreis.
The presented result concerns with the properties of zero subsets of invertible in the sense of Ehrenpreis function $f$.
Namely, we establish some properties
of the zero subset formed by  zeros of $f$ laying  not far from the real axis.

\section{\large Introduction}
 
Let $\mathcal E'$ denote the strong dual to the Fr\'echet  space $\mathcal E:=C^{\infty} (\mathbb R)$. 
Recall that Fourier-Laplace transform operator acting in $\mathcal E$ is defined  by the formula 
 		$$\psi (z)=S(e^{-\mathrm{i}tz}),\quad S\in\mathcal E',$$
  and the image $\mathcal P$ of $\mathcal E'$ under this transform
becomes
the topological algebra ({\sl Schwartz algebra}) if we  equip it   
 with the topology and the algebraical structure induced from $\mathcal E'$.
It is well-known that $\mathcal P$ consists of all entire functions 
of exponential type  having at most polynomial growth along the real axis  \cite[Theorem 7.3.1]{Horm}.

  {\sl The division theorem} is valid for $\varphi\in\mathcal P$ if the following implication holds:
 $$\Phi\in\mathcal P, \quad \Phi/\varphi\in Hol (\mathbb C) \Longrightarrow \Phi/\varphi\in\mathcal P.$$

Now, we explain why having this property for $\varphi\in\mathcal P$ is important for the applications.

In  \cite{Ehren} L. Ehrenpreis establishes that 
the validity of the division theorem for  $\varphi\in\mathcal P$
 is equivalent to 
  the invertibility in the spaces $\mathcal E$ 
and  $\mathcal D'=(C_0^{\infty} (\mathbb R))'$ of the  distribution $S=\mathcal F^{-1} (\varphi)$,
which means 
$$
S*\mathcal E=\mathcal E, 
$$
$$
S*\mathcal D' =\mathcal D',
$$
where the symbol $*$ denotes the convolution.

We say that $\varphi\in\mathcal P$ 
is
 {\sl invertible in the sense of Ehrenpreis} 
if the division theorem is valid for it
  (see \cite{NF-ZP}, \cite{NF-LJM}).
	
Below, we will use  the following analytical criterion due to L. Ehrenpreis \cite[Theorems I, 2.2, Proposition 2.7]{Ehren}:
  $\varphi \in\mathcal P$ is invertible (in the sense of Ehrenpreis) if and only if it  
is {\sl slowly decreasing}, i.e.
 there exists  $ a>0$ such that
\begin{equation}\label{treb}
\forall x\in\mathbb R\ \exists x'\in\mathbb R: \ |x-x'|\le a\mathrm{ln}\, (2+|x|), \
|\varphi (x')|\ge (a+|x'|)^{-a} .
\end{equation}

Each $S\in\mathcal E'$ generates  the convolution operator $M_S$ acting in $\mathcal E:$
$$
M_S(f)=S*f,\quad f\in\mathcal E.
$$
It is easy to check that
 $\varphi\in\mathcal P$ is invertible in the sense of Ehrenpreis  if and only if 
the convolution operator $M_S$ generated by 
$S=\mathcal F^{-1}(\varphi)$ 
is surjective.

Let $\varphi\in\mathcal P$ be invertible in the sense of Ehrenpreis, $\Lambda$ be its zero set,
$\Lambda'\subset\Lambda$.
Then,  $(\mathrm{i}\Lambda')$  coincides with the spectrum
 of the differentiation-invariant subspace $W\subset\mathcal E$ 
which has the following property: each $f\in W$ is represented as a series with grouping
of exponential monomials contained in $W$.
This fact is established in 
\cite{Ehren}, \cite{Ber-Tayl}, \cite{Ber-Str}
for the subspaces of the form
$$
W=\{f\in\mathcal E:\quad S*f=0\},
$$
where $S\in\mathcal E'$ is fixed.
And it is also true for general  subspaces admitting (weak) spectral synthesis with respect to the differentiation operator.

Summarizing the above, we may conclude that there are enough reasons to study the behavior and zero subsets
of functions $\varphi\in\mathcal P$ which are invertible in the
sense of Ehrenpreis.

We establish some  restrictions on the distribution of  real parts of  zeros
lying not far from the real axis (Theorem \ref{tm-1}).
The result  generalizes \cite[Lemma 2]{NF-LJM} and  \cite[Proposition 6.1]{Ehren}).


\section{\large Zero sets}

Let
 $\mathcal M=\{\mu_j\},$ $\mu_j=\alpha_j+\mathrm{i}\beta_j,$
$$ 
0<|\mu_1|\le |\mu_2|\le\dots,
$$
be such that
 $\beta_j=O(\mathrm{ln}\,|\mu_j|)$ as $j\to\infty,$
and the formula
\begin{equation}
\psi(z)=\lim_{R\to\infty} \prod_{|\mu_j|\le R}\left( 1-\frac{z}{\mu_j}\right)
\label{fi-lam}
\end{equation}
defines  entire function of exponential type.

In \cite[Lemma 1]{NF-LJM},
we established the following fact.

\smallskip

\noindent
{\bf Lemma A.}
{\it 
The function
$\psi$ defined by (\ref{fi-lam}) is invertible in the sense of Ehrempreis element of the Schwartz  algebra
 $\mathcal P$
if and only if the same is true about the function 
\begin{equation}\label{psi-1}
\psi_1(z)=\lim_{R\to\infty} \prod_{|\alpha_j|\le R}\left( 1-\frac{z}{\alpha_j}\right) .
\end{equation}
}

\smallskip

Taking into account Lemma A,  we give  another formulation of  \cite[Lemma 2]{NF-LJM}.
 \smallskip

\noindent
{\bf Lemma B} 
{\it Let $\psi\in\mathcal P$ be  invertible in the sense of Ehrenpreis,  $\mathcal M=\{ \mu_k\}\subset\mathbb C$
be its zero set.

Then, for  the subset $\mathcal M'\subset\mathcal M$  defined
by 
$$
\mu_k\in\mathcal M' \Longleftrightarrow |\mathrm{Im}\, \mu_k|\le M_0\mathrm{ln}\,|\mathrm{Re}\, \mu_k|
$$
with  $M_0>0$ fixed,
$$
\varlimsup_{|x|\to\infty} \frac{m_{\mathrm{Re}}(x,1)}{\mathrm{ln}\, |x|} <\infty ,
$$
where $m_{\mathrm{Re}}(x,1)$ denotes the number of points of the sequence
$$\mathrm{Re}\, \mathcal M'=\{\mathrm{Re}\, \mu_k:\ \mu_k\in\mathcal M'\} $$
 contained in the segment $[x-1;x+1].$ }

\smallskip

Consider non-decreasing function
 $l:[0;+\infty)\to [1;+\infty)$ satisfying 
 \begin{equation}\label{cond-1}
\mathrm{ln}\, t=O(l(t)),\quad t\to\infty,
\end{equation}
\begin{equation}\label{cond-2}
\varlimsup_{t\to+\infty} \frac{\mathrm{ln}\, l(t)}{\mathrm{ln}\, t}<\frac{1}{2} ,
\end{equation}
and
\begin{equation}\label{cond-3}
\varlimsup_{t\to+\infty} \frac{l(Kt)}{l(t)}<+\infty 
\end{equation}
for some  $K>1$.
 
The following theorem  generalizes   Lemma B.

\begin{theorem} \label{tm-1}
 Let
$\psi\in\mathcal P$  be invertible in the sense of Ehrenpreis with the zero set $\mathcal M=\{ \mu_k\}$,
and  $\mathcal M'\subset\mathcal M$ be defined
by 
$$
\mu_k\in\mathcal M' \Longleftrightarrow |\mathrm{Im}\, \mu_k|\le M_0 \cdot l\left(|\mathrm{Re}\, \mu_k|\right) 
$$
for a fixed  $M_0>0.$

Then, 
\begin{equation}\label{star-001}
\varlimsup_{|x|\to\infty} \frac{m_{\mathrm{Re}}(x,1)}{l ( |x|)} <\infty ,
\end{equation}
where $m_{\mathrm{Re}}(x,1)$ denotes the number of points of the sequence
$$
\mathrm{Re}\, \mathcal M'=\{\mathrm{Re}\, \mu_k:\ \mu_k\in\mathcal M'\} 
$$
 contained in the segment $[x-1;x+1].$ 
\end{theorem}

First, we prove the following auxiliary proposition.

\begin{proposition} \label{pr-1}
Let $\psi,$  $\mathcal M,$ 
$\mathcal M'$ 
be the same  as in Theorem
   \ref{tm-1},
$$
\mathcal M''=\mathcal M\setminus\mathcal M',\quad
\alpha_j=\mathrm{Re}\, \mu_j .
$$

Then, the function
\begin{equation}\label{psi-1-1}
\psi_1(z)=\lim_{R\to\infty}
\left( \prod_{\substack{|\mu_j|\le R\\ \mu_j\in\mathcal M'}}\left( 1-\frac{z}{\alpha_j}\right)
\prod_{\substack{|\mu_j|\le R\\ \mu_j\in\mathcal M''}}\left( 1-\frac{z}{\mu_j}\right) \right) 
\end{equation}
belongs to the algebra $\mathcal P$, 
and there exists  $M_1>0$ such that 
 \begin{equation}\label{10-star}
 \forall x\in\mathbb R, |x|>2,\ \exists z'\in\mathbb C:\ |z'-x|\le M_1 l(|x|)\quad \text{and} 
\quad 
\mathrm{ln}\, |\psi_1 (z')|\ge -M_1 l(|z'|).
\end{equation}
\end{proposition}

\begin{proof}
We start with estimating the single multiplier  $\left( 1-\frac{x}{\alpha_j}\right),$ $x\in\mathbb R,$
where $\alpha_j=\mathrm{Re}\,\mu_j,$ $\mu_j\in\mathcal M' .$
It is easy to see that 
\begin{multline*}
\left| 1-\frac{x}{\alpha_j}\right| \le \left| 1-\frac{x}{\mu_j}\right| \left( 1+\frac{M_0^2 l^2(\alpha_j)}{\alpha_j^2}\right)^{1/2}\le\\
\le \left| 1-\frac{x}{\mu_j}\right| \left( 1+O\left(\frac{l^2(\alpha_j)}{\alpha_j^2}\right)\right).
\end{multline*}

 Taking into account (\ref{cond-2}), we get
\begin{equation}\label{psi-psi-1}
\mathrm{ln}\, |\psi_1 (x)| \le \mathrm{const}\, \mathrm{ln}\, |\psi (x)|,\quad x\in\mathbb R.
\end{equation}
Hence, $\psi_1\in\mathcal P$.

To estimate  $|\psi_1|$ from below,
 we consider the auxiliary function 
$$
\psi^+ (z)=\lim_{R\to\infty} \left( \prod_{\substack{|\mu_j|\le R\\ \mu_j\in\mathcal M'_{-}\bigcup\mathcal M''}}
\left( 1-\frac{z}{\mu_j}\right)\prod_{\substack{|\mu_j|\le R\\ \mu_j\in\mathcal M'_{+}}}
\left( 1-\frac{z}{\alpha_j}\right)\right),
$$
where
$$
\mathcal M'_{+}=\{\mu_j\in\mathcal M':\ \beta_j\ge 0\},\quad \mathcal M'_{-} =\mathcal M'\setminus\mathcal M'_{+}.
$$

Because of  (\ref{cond-3}), we have
\begin{equation}\label{cond-3-1}
l(Kt)\le C_0 l(t),\quad t\ge 0,
\end{equation}
for some $C_0>0.$
Notice that
$$
\left| 1-\frac{z}{\mu_j}\right| \le \frac{|\mu_j-z|}{|\alpha_j|}=
\frac{\left( (\alpha_j-x)^2 +(\beta_j -y)^2\right)^{1/2}}{|\alpha_j|},
$$
where
$\mu_j=\alpha_j+\mathrm{i}\beta_j\in\mathcal M'_{+} ,$ $z=x+\mathrm{i} y.$
Together with (\ref{cond-3-1}), it gives us the inequality
\begin{equation}\label{mu-2}
\left| 1-\frac{z}{\mu_j}\right| \le \left| 1-\frac{z}{\alpha_j}\right| 
\end{equation}
if  $z=x+2\mathrm{i} C_0M_0 l(|x|)$, $\mu_j\in\mathcal M'_{+}$ and $|\mathrm{Re}\, \mu_j|=|\alpha_j|\le 4K|x|.$

From the other hand, for $z=x+2\mathrm{i} C_0M_0 l(|x|)$,
$\mu_j=\alpha_j+\mathrm{i}\beta_j\in\mathcal M'_{+}$ 
we have
\begin{equation}\label{mu-3}
\left| 1-\frac{z}{\mu_j}\right| \le \left| 1-\frac{z}{\alpha_j}\right| \cdot \left( 1+\frac{C_1 l^{2}(\alpha_j)}{\alpha_j^2}\right)^{1/2} 
\end{equation}
if $|\alpha_j|> 4K|x|,$
where the constant $C_1>0$ depends only on $l$ and $M_0 .$

The relations (\ref{cond-2}), (\ref{mu-2}), (\ref{mu-3}) lead to the estimate
\begin{equation}\label{ots-1}
\mathrm{ln}\, |\psi (z)| \le \mathrm{const}\, \mathrm{ln}\, |\psi^+ (z)| +O(1) \quad\text{as}\quad |x|\to\infty 
\end{equation}
if $z=x+2\mathrm{i} C_0M_0 l(|x|)$.

By the similar way, we get that
\begin{equation}\label{ots-2}
\mathrm{ln}\, |\psi^+ (z)| \le \mathrm{const}\, \mathrm{ln}\, |\psi_1 (z)| +O(1),\quad \text{as}\quad |x|\to\infty ,
\end{equation}
where $z=x-2\mathrm{i} C_0M_0 l(|x|).$

Applying the analytical criterion(\ref{treb}) 
and the minimum modulus theorem \cite[Ch. 1, Sec. 8, Th. 11]{Lev-Distr},
we arrive to the estimate 
$$
 \mathrm{ln}\, |\psi (z)| \ge -C_2\mathrm{ln}\, |x|
 $$
for all 
$z:$ $|z-x|=C_2\mathrm{ln}\, |x|$,
$x\in\mathbb R,$ $|x|>2$ and some $C_2>0$.

Notice that $\psi\in\mathcal P$ implies the inequality
\begin{equation}\label{alg-P}
|\psi (z)|\le C_{\psi} (2+|z|)^{C_{\psi}}e^{C_{\psi}|\mathrm{Im}\, z |}
\end{equation}
with some $C_{\psi}>0.$

Fix $x\in\mathbb R,$ $|x|>2,$
set $z_x=x+\mathrm{i} C_2\mathrm{ln}\, |x|$ and then apply the minimum modulus theorem
to the function
$\frac{\psi}{\psi (z_x)}$
in the disc $|z-z_x|\le 4 C_0M_0 l(|x|) .$
Taking into account (\ref{cond-1}), (\ref{cond-3})  and (\ref{alg-P}),
we find  $M_2>0$ and $ \theta\in (2;4)$ such that
\begin{equation}\label{w-x}
\mathrm{ln}\, |\psi (z)| \ge -M_2 l(|x|)\quad \text{if} \quad  |z-z_x|=\theta C_0M_0 l(|x|) .
\end{equation}
Further, from (\ref{ots-1}), (\ref{w-x}) and (\ref{cond-2})--(\ref{cond-3}),
it follows that there exists   $M_3>0$ with the property:
$$\forall x\in\mathbb R, \ |x|>2,\quad  
 \exists w_x\in\mathbb C\quad \text{such that}
$$
 $$
|w_x-x|\le M_3 l(|x|)\quad\text{and}\quad
\mathrm{ln}\, |\psi^+ (w_x)|\ge -M_3 l (|w_x|).
$$

Now, we apply the above argument including the minimum modulus theorem
to the function $\frac{\psi^+}{\psi^+(w_x)}$ and the disc $|z-w_x|\le R,$ 
where  
$$
R=\max\{ 2M_3l(|x|), \ 4C_0M_0 l(|x|)\}.
$$
It gives us  the estimate  for $\mathrm{ln}\, |\psi^+ (z)|$  which is similar to (\ref{w-x}).
This estimate and (\ref{ots-2}), together with (\ref{cond-2})--(\ref{cond-3}), lead  us to
 the  assertion.
\end{proof}

\begin{remark}
{\rm It is not difficult to see
that applying the minimum modulus theorem
one more time, to the function $\frac{\psi_1}{\psi(z')},$
we obtain  the following version of  Proposition \ref{pr-1}:

\noindent
there exists $M_1>0$ such that
 $$
\exists M_1>0:\ \forall x\in\mathbb R, |x|>2,\ \exists x'\in\mathbb R:
$$
\begin{equation}\label{r-1}
 |x'-x|\le M_1 l(|x|)\quad \text{and} \quad 
\mathrm{ln}\, |\psi_1 (x')|\ge -M_1 l(|x'|).
\end{equation}
}
\end{remark}

\noindent
{\bf Proof of Theorem \ref{tm-1}.}
 
\noindent
{\sl Wlog}, we  assume that 
 $\psi$ 
is bounded on the real axis and its exponential type equals 1.
Because of (\ref{psi-psi-1}), we may also assume that the same is true for the function $\psi_1$
defined by  (\ref{psi-1-1}).
Let  $\lambda_k\in\mathcal M'$ and  $\alpha_k=\mathrm{Re}\, \lambda_k .$

Further in the proof,  we use  symbol  $\psi$ to denote the function  $\psi_1$ defined in (\ref{psi-1-1}).

If (\ref{star-001}) fails then 
 \begin{equation}\label{opp-1}
\lim_{j\to\infty} \frac{m_{\mathrm{Re}}(x_j,1)}{l ( |x_j|)} =\infty 
\end{equation}
 for some $x_j$,  $|x_j| \to \infty .$
For clarity, we assume that $x_j>0$.

Consider entire functions 
$$
\psi_j (z)=\psi(z)(z-x_j)^{m_j}\cdot\prod\limits_{k: |\alpha_k-x_j|\le 1}(z-\alpha_k)^{-1}, 
$$
where $m_j =m(x_j ,1),$ $j=1,2,\dots $
It is easy to check that the estimates   
$$
\sup\limits_{x\in\mathbb R} |\psi_j(x)|
\le C_0 2^{m_j},  \quad j=1,2,\dots,
$$
hold
with  $C_0=\max\{ \sup\limits_{t\in\mathbb R} |\psi (t)|,\, 1\}.$
By   
 Bernstein's theorem  \cite[Chapter 11]{Boas}, the  inequalities
\begin{equation}\label{est-deriv}
\sup\limits_{x\in\mathbb R} |\psi^{(n)}_j (x)|\le C_0 2^{m_j} 
\end{equation}
are also valid  for all $n, \, j\in\mathbb N.$

From the Taylor expansion of  $\psi_j$ at
   $x_j$ and the estimates (\ref{est-deriv}), it follows  that
$$
|\psi_j(z)|\le C_02^{m_j} (m_j!)^{-1} |z-x_j|^{m_j}e^{|z-x_j|}, \quad z\in\mathbb C.
$$
Hence, 
for all $x\in\mathbb R$ satisfying the condition
\begin{equation}\label{star}
\mathrm{ln}\, C_0 +|x-x_j|+m_j\mathrm{ln}\, |x-x_j|-\mathrm{ln}\, (m_j!)\le -n \, l( x_j )-m_j\mathrm{ln}\, 2 ,
\end{equation}
where $n\in\mathbb N,$
we have the inequality 
\begin{equation}\label{ca-0}
|\psi_j(x)|\le x_j^{-n}\cdot 2^{-m_j} .
\end{equation}
 
By Stirling's formula,  the relation (\ref{star}) will follow from the inequality
$$
|x-x_j|+m_j\mathrm{ln}\, |x-x_j|-m_j\mathrm{ln}\, m_j \le -n\, l ( x_j) -C_1 m_j,
$$
where $C_1$ is an absolute constant.

Because of  (\ref{opp-1}), for each $n\in\mathbb N$ we can find   $j_n$ such that
\begin{equation}\label{cross}
-n\, l( x_j)\ge -m_j , \quad j=j_n, j_n+1,\dots 
\end{equation}
Fix $b\in (0;e^{-C_1-2}) .$
The estimates 
(\ref{ca-0}) are valid for $j\ge j_n$ and  all $x\in\mathbb R$  such that
$|x-x_j|\le b m_j .$
This fact, inequalities  (\ref{cross}) and the relations
$$
|\psi (z)|\le 2^{m_j}|\psi_j(z) |,\quad z\in\mathbb C,\ \ j=1,2, \dots ,
$$
  imply 
$$
|\psi (x)|\le e^{-n l(x_j)}
$$
for $x\in\mathbb R$ satisfying $ |x-x_j|\le b n\, l( x_j),\ \ j\ge j_n,\ \ n\in\mathbb N.$
It means that  $\psi$ does not satisfy  (\ref{r-1}), i.e. 
the assumption (\ref{opp-1}) leads to the contradiction.

Q.E.D.

\section*{\large Acknowledgments} 

\grant{This work is supported 
by Ministry  of Science and  Higher education of Russian Federation (code of scientific theme FZWU-2020-0027).}


\def\bibname{\vspace*{-30mm}{

 \end{document}